\documentclass{amsart}
\usepackage{marginnote}
\usepackage{amstext,amscd,xcolor}
\usepackage{bbm}
\usepackage{dsfont}
\usepackage{amsthm, enumerate}
\usepackage{amsfonts}
\usepackage{amssymb}
\usepackage{dsfont}
\usepackage{amsmath}
\usepackage{tikz}
\usepackage{setspace}

\usepackage{todonotes}
\usepackage{graphicx}
\usepackage{mathrsfs}
\newtheorem{mtheorem}{Theorem}
\newtheorem{theorem}{\textbf{Theorem}}[section]

\newtheorem{corollary}{\textbf{Corollary}}

\newtheorem{lemma}[theorem]{\textbf{Lemma}}
\newtheorem{proposition}[theorem]{\textbf{Proposition}}
\newtheorem{remark}[theorem]{\textbf{Remark}}

\newtheorem*{question*}{\textbf{Question}}

\theoremstyle{plain}

\newtheoremstyle{example}% name
  {}% space above
  {}% space below
  {}% body font
  {}% indent amount
  {\itshape}% theorem head font
  {.}% punctuation after theorem head
  {.5em}% space after theorem head
  {\thmname{#1}\thmnumber{ #2}\thmnote{ (#3)}}% theorem head spec
\theoremstyle{example}
\newtheorem{example}{Example}[section]

\numberwithin{equation}{section}

\begin{document}

\title[Stationary random iterations of Lipschitz maps]{Stationary random iterations of Lipschitz maps: a generalized displacement condition for the existence of invariant structures}

\author[B. Figliaggi ]{Bella Figliaggi }
\address{Instituto de Ciências Matemáticas e Computação, Universidade de São Paulo}
\email{bellafigliaggi@usp.br}

\author[E. Matias]{Edgar Matias}
\address{Departamento de Matem\' atica, Universidade Federal da Bahia, Av. Adhemar de Barros s/n, 40170-110 Salvador, Brazil}
\email{edgar.matias@ufba.br}

\author[T. Pereira]{Tiago Pereira}
\address{Instituto de Ciências Matemáticas e Computação, Universidade de São Paulo}
\address{Department of Mathematics, Imperial College London, London, UK}
\email{tiago@icmc.usp.br}

%%%

\begin{abstract}
We study random iterations generated by infinitely many Lipschitz maps on non-compact complete separable metric spaces. We generalize classical displacement conditions ensuring the existence and uniqueness of invariant structures.  We prove two main results.  The first uses the graph transform operator to extend Elton's condition \cite{Elton1990} for the existence and uniqueness of invariant graphs.  The second uses the Markov operator to generalize a displacement condition for stationary measures, introduced by Diaconis and Freedman~\cite{DiaconisFreedman1999} and by Kloeckner \cite{Kloeckner2022}. We present several examples illustrating the applicability of our results.
\end{abstract}

\maketitle

\section{Introduction}

Let $(\Omega,\mathcal F,\mathbb{P})$ be a probability space and consider measurable maps $\theta\colon \Omega\to\Omega$ and $f\colon \Omega\times M\to M$, where $(M,d)$  is a complete separable metric space endowed with its Borel \(\sigma-\)algebra $\mathcal B(M)$, and \(\theta\) is assumed to be measurably invertible.  Consider the following problem: find a measurable map $\pi\colon \Omega\to M$ satisfying the  equation
\[
f(\omega,\pi(\omega))
= 
\pi(\theta(\omega)) \quad \quad \mbox{a.e.}
\]
A solution of this equation is called an \emph{invariant graph}. Such objects arise in several areas of dynamical systems, including \emph{iterated function systems} \cite{Stenflo2012}, \emph{fractal geometry} \cite{barnsley1989,barnsley2008}, \emph{forced systems} \cite{HADJILOUCAS_NICOL_WALKDEN_2002,JAGER_2007}, \emph{heterogeneous random networks} \cite{Bian2025}, and \emph{random walks on the circle} \cite{SALCEDO2026,malicet2017}.

We study this equation through the associated random iteration of maps. Namely, for each $\omega\in\Omega$, set $f_\omega(\cdot)=f(\omega,\cdot)$ and define
$$
f_\omega^n
=
f_{\theta^{n-1}(\omega)}
\circ
\cdots
\circ
f_\omega.
$$
The family of compositions $(f_\omega^n)_{(\omega,n)\in\Omega\times \mathbb N}$, together with the probability measure $\mathbb P$, defines a \emph{random iteration of maps} (RIM), which we  denote by the triple $(f,\theta,\mathbb P)$.\footnote{Throughout this paper, whenever no confusion can arise, we omit the underlying spaces and their $\sigma$-algebras. Also, the space $\Omega\times M$ is endowed with the product $\sigma$-algebra.}

We focus on \emph{stationary} random iterations of maps, that is, RIMs for which $\mathbb{P}$ is $\theta$-invariant. From now on, we assume that the maps $f_\omega$ are Lipschitz. This regularity assumption, together with the separability of $M$\footnote{The separability of $M$ ensures that $\mathcal{B}(M\times M)=\mathcal{B}(M)\otimes\mathcal{B}(M)$, which may fail for arbitrary metric spaces. Together with the continuity of each $f_\omega$, this guarantees the measurability of the Lipschitz constant $\omega\mapsto\|f_\omega\|$.} provides a well-established framework for the study of invariant structures \cite{Elton1987,Stenflo2012,steinsaltz1999}. For a Lipschitz map $g\colon M\to M$, we denote its Lipschitz constant by
\[
\|g\|=\sup_{x\neq y}\frac{d(g(x),g(y))}{d(x,y)}.
\]

We restrict our analysis to stationary RIMs $(f,\theta,\mathbb{P})$ satisfying the integrability condition%
\begin{equation}\label{king}
\int \log^{+}\|f_\omega\|\,d\mathbb{P}(\omega)<\infty.
\end{equation}
Condition \eqref{king} is the standing integrability assumption ensuring the existence of a $\theta$-invariant measurable function $\lambda\colon \Omega\to[-\infty,\infty)$ such that
\[
\lambda(\omega)
=
\lim_{n\to\infty}\frac1n\log\|f_\omega^n\|
\]
for $\mathbb P$-almost every $\omega$. We call $\lambda$ the \emph{maximal Lyapunov exponent} of the RIM $(f,\theta,\mathbb P)$.  A theorem of Elton \cite{Elton1990} asserts that a stationary RIM with negative maximal Lyapunov exponent admits a unique invariant graph provided that, for some $x_{0}\in M$, the displacement integrability condition
\begin{equation}\label{oldDIcondition}
\int \log^{+} d(x_{0},f_{\omega}(x_{0}))\, d\mathbb{P}(\omega)<\infty
\end{equation}
holds. More precisely, under these assumptions, for $\mathbb{P}$-almost every $\omega$, the limit
\[
\pi(\omega)
=
\lim_{n\to\infty}
f_{\theta^{-1}(\omega)}\circ \cdots\circ f_{\theta^{-n}(\omega)}(x_{0})
\]
exists and defines the unique invariant graph of the RIM.

This representation of the invariant graph implies that, for every $x\in M$, the laws\footnote{Given a random element $X\colon \Omega\to M$, the \emph{law} of $X$ is the pushforward measure $X_{*}\mathbb P$.} of the random variables $f_\omega^n(x)$ converge to the law of $\pi$. Observe that condition \eqref{oldDIcondition} is automatically satisfied for RIMs generated by finitely many maps or whenever $M$ is bounded.

However, the displacement condition~\eqref{oldDIcondition} fails to cover several important RIMs with heavy-tailed displacements, as illustrated by Examples~\ref{ex:translationsandcontractions} and~\ref{exampleiid}. In this paper, we generalize~\eqref{oldDIcondition} by introducing an intrinsic displacement condition that allows us to establish analogous results for systems not covered by Elton's classical theorem.
To formulate our hypothesis, we consider the graph transform $\Gamma$ acting on measurable maps $\psi\colon \Omega\to M$, defined by
\[
\Gamma\psi(\omega)
=
f_{\theta^{-1}(\omega)}
\bigl(\psi(\theta^{-1}(\omega))\bigr).
\]
A measurable map $\varphi\colon \Omega\to M$ is called a \emph{reference map} for a stationary RIM $(f, \theta,\mathbb{P})$ if 
\[
\int \log^{+} d\bigl(\varphi(\omega),\Gamma\varphi(\omega)\bigr)\,
d\mathbb{P}(\omega)<\infty.
\]
\begin{mtheorem}[Generalized displacements]
\label{Eltongener}
Let $(f,\theta,\mathbb{P})$ be a stationary RIM with negative maximal Lyapunov exponent. Assume that the set of reference maps is non-empty. Then, there is a unique (a.e.) invariant graph $\pi\colon \Omega\to M$ and  for every reference map $\varphi \colon \Omega\to M$ we have 
$$
\pi(\omega)
=
\lim_{n\to\infty}
f_{\theta^{-1}(\omega)}
\circ
\cdots
\circ
f_{\theta^{-n}(\omega)}
\bigl(\varphi(\theta^{-n}(\omega))\bigr)
$$
for $\mathbb P$-almost every $\omega$. Moreover, the sequence of laws of the random elements
\[
\omega\mapsto f_{\omega}^{n}(\varphi(\omega))
\]
converges to the law of $\pi$.
\end{mtheorem}

By setting the reference map as a constant, $\varphi(\omega)=x_{0}$, we recover the classical displacement condition~\eqref{oldDIcondition}. Indeed, by the $\theta$-invariance of $\mathbb{P}$:
\[
\int \log^{+}d(x_{0},f_{\theta^{-1}(\omega)}(x_{0}))\, d\mathbb{P}(\omega)
=
\int \log^{+}d(x_{0},f_{\omega}(x_{0}))\, d\mathbb{P}(\omega).
\]

Theorem~\ref{Eltongener} shows that the relevant integrability condition is not necessarily an absolute displacement condition based on a fixed point, but rather a relative displacement condition with respect to a suitable measurable map. This allows us to treat systems for which no point has integrable displacement, while a unique invariant graph still exists.

\begin{remark}
\emph{Let $\varphi \colon \Omega \to M$ be a reference map and  $\psi \colon \Omega \to M$ be measurable satisfying
$$
\int \log^{+}d(\varphi(\omega),\psi(\omega))\, d \mathbb{P}(\omega)<\infty.
$$
Then, \(\psi\) is also a reference map. This holds because the above condition defines an equivalence relation $\varphi \sim \psi$ on the space of measurable maps. This relation is preserved by the graph transform $\Gamma$. That is, if $\varphi \sim \psi$, then $\Gamma \varphi \sim \Gamma \psi$. Since $\Gamma \varphi \sim \varphi$ (as $\varphi$ is a reference map), transitivity yields $\Gamma \psi \sim \psi$, which is precisely the defining condition for $\psi$ to be a reference map. Therefore, Theorem~\ref{Eltongener} holds for any map in the equivalence class of $\varphi$.}
\end{remark}

We also apply Theorem \ref{Eltongener} in the study of independent and identically distributed random iterations of maps to show the existence and uniqueness of stationary measures in situations where Elton's result does not apply. Namely, let $E$ be a standard Borel space and let $\nu$ be a probability measure on $E$. Set
\[
\Omega=E^{\mathbb Z},
\qquad
\mathbb P_{\nu}=\nu^{\mathbb Z}.
\]
Let $\sigma\colon \Omega\to\Omega$ be the shift map. Given a measurable map $f\colon E\times M\to M$, we use the same notation to define $f\colon \Omega\times M\to M$ by
\[
f_\omega=f_{\omega_0}.
\]
The stationary RIM $(f,\sigma,\mathbb P_\nu)$ is called an \emph{independent and identically distributed (i.i.d.)} random iteration of maps. Since the shift map  depends only on $E$, we shall denote this system by $(f,\mathbb{P}_{\nu})$. A stationary measure is a fixed point of the  \emph{Markov operator} 
\[
\mathcal T\mu
=
\int f_{\alpha*}\mu\,d\nu(\alpha)
\]
acting on the space of probability measures on $M$.

Let $\mathcal{F}_{-\infty}$ be the $\sigma$-algebra generated by the canonical projections
$$
\Pi(\omega)=\omega_{-n}, \quad\quad n\geq 1.
$$
Similarly, we define $\mathcal F_{+\infty}$ using the nonnegative coordinates.

\begin{corollary}\label{coro:letac1}
Let $(f,\mathbb{P}_{\nu})$ be an i.i.d. RIM  with negative maximal Lyapunov exponent. Assume that the set of reference maps is non-empty. Let $\pi$ be the unique invariant graph given by Theorem~\ref{Eltongener}.  Then, the law of $\pi$  is the unique stationary measure. Moreover, for every  $\mathcal{F}_{-\infty}$-measurable   reference map $\psi\colon \Omega\to M$ we have 
\[
\mathcal T^{n}(\psi_{*}\mathbb{P}_{\nu})
\longrightarrow
\pi_{*}\mathbb P_{\nu}
\]
in the weak-$*$ topology.
\end{corollary}

In Example~\ref{exampleiid}, we illustrate the corollary with an i.i.d. RIM featuring a non-integrable fixed-point displacement.
The preceding theorem establishes convergence of the iterates of the Markov operator to the stationary measure, without providing a rate of convergence. Under additional assumptions of contraction on average, however, quantitative convergence rates can be obtained.

The study of random iterations satisfying average contractivity conditions goes back to Barnsley and Elton \cite{BarnsleyElton1988}. More precisely, given $q>0$ and $r\in(0,1)$, we say that an i.i.d. RIM $(f,\mathbb{P}_{\nu})$ is \emph{$r$-contracting on $L^{q}$-average between points} if

$$
\int d(f_{\alpha}(x),f_{\alpha}(y))^{q}\,d\nu(\alpha)
\leq r\,d(x,y)^{q}
$$
for every $x,y\in M$. When $\nu$ is supported on a finite set, Barnsley and Elton proved the existence of a unique attractive stationary measure by establishing convergence of the ``unilateral'' reverse-order iteration
\begin{equation}\label{eq:unilateral}
f_{\omega_0}\circ\cdots\circ f_{\omega_n}(x_0),
\qquad n\geq 0.
\end{equation}

Subsequently, Diaconis and Freedman \cite{DiaconisFreedman1999} studied i.i.d. RIMs satisfying
\begin{equation}\label{diaconis}
\int \|f_{\alpha}\|\,d\nu(\alpha)<\infty,
\qquad
\int \log \|f_{\alpha}\|\,d\nu(\alpha)<0,
\end{equation}
together with the displacement condition
\begin{equation*}
\int d(x_0,f_\alpha(x_0))\,d\nu(\alpha)<\infty
\end{equation*}
for some $x_0\in M$. They established the existence of a unique attractive stationary measure, by proving convergence of \eqref{eq:unilateral}. Moreover, they obtained an exponential convergence rate to the stationary measure in the Prokhorov metric.

The assumptions \eqref{diaconis} imply that the system is $r$-contracting on $L^{q}$-average between points for some $q>0$ and $r\in(0,1)$; see \cite[Lemma 3.9]{Kloeckner2022}. Thus, the systems studied by Diaconis and Freedman also fall within the framework of contraction on average introduced above.

More recently, Kloeckner \cite{Kloeckner2022} used the Wasserstein metric to extend the results of Barnsley--Elton \cite{BarnsleyElton1988} and Diaconis--Freedman \cite{DiaconisFreedman1999}. For RIMs that are $r$-contracting on $L^{q}$-average between points, he established existence and uniqueness of the stationary measure, together with convergence rates in Wasserstein distance, under the displacement condition
\begin{equation}\label{disconditionKloc}
\int d(x_0,f_\alpha(x_0))^{q}\,d\nu(\alpha)<\infty
\end{equation}
for some $x_0\in M$; see \cite[Theorem~3.1]{Kloeckner2022}.

An interesting feature of Kloeckner's approach is that the existence and uniqueness of the stationary measure are established without relying on convergence of the reverse-order iteration. We now present a slight generalization of his result.

To describe the Wasserstein framework, we first recall the definition of a transport plan. Given two Borel probability measures \(\mu_{1}\) and \(\mu_{2}\) on \(M\), a \emph{transport plan} between \(\mu_{1}\) and \(\mu_{2}\) is a probability measure \(\Upsilon\) on \(M\times M\) such that
\[
\Pi_{1*}\Upsilon=\mu_{1}
\quad\text{and}\quad
\Pi_{2*}\Upsilon=\mu_{2},
\]
where \(\Pi_{1},\Pi_{2}\colon M\times M\to M\) are the canonical projections. We denote by \(\Gamma(\mu_{1},\mu_{2})\) the set of all transport plans between \(\mu_{1}\) and \(\mu_{2}\). Denote by $\mathcal{P}_{1}(M)$ the space of all Borel probability measures on $M$. For $q> 0$, we define the \emph{extended Wasserstein metric} \(W_{q}\colon \mathcal{P}_{1}(M)\times \mathcal{P}_{1}(M)\to[0,\infty]\) by
\[
W_{q}(\mu_{1},\mu_{2})
=
\left(
\inf_{\Upsilon\in\Gamma(\mu_{1},\mu_{2})}
\int d(x,y)^{q}\,d\Upsilon(x,y)
\right)^{\min\{1,1/q\}}.
\]
By extended metric, we mean that $W_{q}$ has all properties of a genuine metric, but is allowed to take  the value \(+\infty\).  See Villani \cite{Villani2009}. Observe that
\[
W_{q}(\delta_{x_{0}},\mathcal{T}\delta_{x_{0}})^{\max\{1,q\}}
=\int d(x_{0},f_{\alpha}(x_{0}))^q\, d\nu(\alpha).
\]
This motivates extending condition~\eqref{disconditionKloc} to a measure condition by considering the Wasserstein displacement between a measure and its image under the Markov operator, as in the following:

\begin{mtheorem} \label{tb}
Let $(f,\mathbb{P}_{\nu})$ be an i.i.d. RIM that is $r$-contracting on $L^{q}$-average between points.    Assume that there exists $\eta\in \mathcal{P}_{1}(M)$ such that
\[
W_{q}(\eta,\mathcal{T}\eta)<\infty.
\]
Then, there is a unique stationary measure. Additionally, the Markov operator is a contraction, with rate \(r^{\min\{1,1/q\}}\) on the space of probability measures that are at finite \(W_q-\)distance from their images under the Markov operator.
\end{mtheorem}

In \cite[Theorem~3.1]{Kloeckner2022}, Kloeckner establishes the contraction property on the space of probability measures with finite $q$-th moments, equivalently, those at finite $W_q$-distance from a Dirac measure. The contraction estimate in our setting follows the same coupling argument, but our displacement condition allows us to work beyond this space. Moreover, Theorem~\ref{tb} provides an intrinsic characterization of the space on which the Markov operator acts as a contraction.

We also provide examples showing that our displacement condition applies to systems beyond the finite-moment setting of Kloeckner's theorem. See Example~\ref{Example:toy}.

\begin{remark}
\label{remark:thm2}
\emph{The set of probability measures $\eta$ such that $W_{q}(\eta, \mathcal{T}\eta)<\infty$ coincides with the set of probability measures that is at a finite $W_{q}$-distance from the stationary measure.
}
\end{remark}

\section{Examples}\label{section:examples}

We now present an example showing that the classical displacement condition~\eqref{oldDIcondition} may fail, while the conclusion of Theorem~\ref{Eltongener} still holds by using a non-constant reference map. Our example is a particular case of the following corollary of Theorem~\ref{Eltongener}.

\begin{corollary}
\label{example1}
Consider a stationary RIM $(g,\theta,\mathbb P)$ on the line
$\mathbb R$ such that
\[
g_{\omega}(x)=a(\omega)x+b(\omega),
\]
where $a$ and $b$ are non-negative measurable functions. Assume that the maximal Lyapunov exponent is negative and that
\[
\int
\log^{+}
\bigl(a(\omega)b(\theta^{-1}(\omega))\bigr)
\,d\mathbb P(\omega)<\infty.
\]
Then the conclusions of Theorem~\ref{Eltongener} hold.
\end{corollary}

Indeed, take
\(
\varphi(\omega)=b(\theta^{-1}(\omega)).
\)
Then
\[
\begin{aligned}
|\varphi(\omega)-\Gamma\varphi(\omega)|
&=
\left|
b(\theta^{-1}(\omega))
-
\left(
a(\theta^{-1}(\omega))b(\theta^{-2}(\omega))
+
b(\theta^{-1}(\omega))
\right)
\right| =
a(\theta^{-1}(\omega))b(\theta^{-2}(\omega)).
\end{aligned}
\]
Hence, the existence of a reference map
follows from the hypothesis of the corollary.

\begin{example}\label{ex:translationsandcontractions}

Let us consider an example given by a stationary random iteration consisting of translations and contractions. Consider the probability space $(I,\mathcal{B},m)$, where $I=(0,1)$, $\mathcal{B}$ is the usual Borel $\sigma$-algebra of $I$ and $m$ is the Lebesgue measure. Let $Q\colon(0,1)\to\mathbb R$ be a random variable, defined on $(I,\mathcal B,m)$, with law\footnote{Given a Borel probability measure $\mu$ on $\mathbb R$, set $F_\mu(x)=\mu((-\infty,x])$ and let $Q_\mu$ denote its quantile function: $Q_\mu(t)=\inf\{x\in\mathbb R\colon F_\mu(x)\geq t\}, t\in(0,1).$ Thus, $Q_{*}m=\mu$. This is the classical quantile construction, which allows us to build, from the Lebesgue measure, a random variable with any prescribed law.} $\operatorname{Cauchy}(0,1)$. Let $\sigma$ be the shift map on the product space $\Omega=I^{\mathbb{Z}}$ and let $\mathbb{P}_{m}$ denote  the product measure $m^{\mathbb{Z}}$ on the product space $(I^{\mathbb{Z}},\mathcal{B}^{\bigotimes \mathbb{Z}})$.

Consider the stationary RIM $(g,\sigma,\mathbb{P}_{m})$  with $g$ given  by
\[
g_\omega(x) = a(\omega) x + b(\omega)
\]
where 
$$
 a(\omega)=
\min \{1, e^{-Q(\omega_{-1})}\} 
\quad \mbox{and} \quad b(\omega)=e^{Q(\omega_{0})}.
$$ 

A direct computation gives 
\[
\int
\log^{+}
\bigl(a(\omega)b(\sigma^{-1}(\omega))\bigr)
\,d\mathbb P_{m}(\omega)<\infty.
\]

Furthermore, note that 
\(
\log a(\omega) = - \max \{Q(\omega_{-1}), 0\}
\). 
Because $\mathbb P_m$ is invariant under the shift map $\sigma$, it follows that the maximal Lyapunov exponent is \(-\infty\). On the other hand, 
\[
\int \log^+ b(\omega) \, d\mathbb P_m(\omega) = + \infty.
\]
Thus, the classical displacement condition fails, whereas the condition of Theorem~\ref{Eltongener} holds.

\end{example}

\begin{remark}
\emph{In the i.i.d. RIM setting, with \(a>0\) almost surely, the displacement assumption of Corollary~\ref{example1} implies Elton's condition~\eqref{oldDIcondition}. See~\cite[Theorem~1.6]{Vervaat1979}.}
\end{remark}

The following corollary provides a different class of examples.

\begin{corollary}\label{coro:mqexample}
Consider a stationary RIM $(f,\theta,\mathbb P)$ on the line $\mathbb R$ with $f$ given by
$$
f_{\omega}(x)=L(x)+b(\omega)
$$
where  $L\colon \mathbb{R}\to \mathbb{R}$ is a contraction and  $b\colon \Omega\to \mathbb{R}$ is a measurable function.  Assume that
$$
\int \log^{+}|L(x)|\, d\mu(x)<\infty
$$ 
where $\mu=b_{*}\mathbb{P}$. Then the conclusions of Theorem \ref{Eltongener} hold.
\end{corollary}

Indeed, the maximal Lyapunov exponent is negative. Taking $\varphi(\omega)=b(\theta^{-1}(\omega))$, we have 
\[
\begin{split}
|\varphi(\omega)-\Gamma \varphi(\omega)|
=|b(\theta^{-1}(\omega))-f_{\theta^{-1}(\omega)} (b(\theta^{-2}(\omega)))|
=| L(b(\theta^{-2}(\omega)))|.
\end{split}
\]
Hence, 
$$
\int \log^{+}|\varphi(\omega)-\Gamma\varphi(\omega)|\, d\mathbb{P}(\omega)
\leq \int \log^{+} |L(b(\theta^{-2}(\omega)))|\, d\mathbb{P}(\omega)= \int \log^{+} |L(x)|\, d\mu(x)<\infty
$$
showing that the hypotheses of  Theorem \ref{Eltongener} are satisfied.

Note that finiteness of the integral above is automatically satisfied if \(L\) is bounded. In this case, \(b\) can have any law. We next present an example in which \(L\) is unbounded.  Using Corollary~\ref{coro:mqexample}, we obtain examples in the i.i.d. case for which Theorem~\ref{Eltongener}, Corollary~\ref{coro:mqexample}, and Theorem~\ref{tb} apply, while the classical displacement condition fails.

\begin{example}\label{exampleiid}
Let $(I,\mathcal B,m)$ and $Q$ be as in
Example~\ref{ex:translationsandcontractions}.
Consider the i.i.d. RIM $(f,\mathbb{P}_{m})$ with
$$
f_{\omega}(x)= L(x) + b(\omega)
\quad \mbox{where} \quad 
L(x) = \frac{1}{2}\log(1+|x|) 
\quad \text{and} \quad 
b(\omega) = e^{Q(\omega_{0})}.
$$
Since $Q$ has  law $\mbox{Cauchy}(0,1)$, denoting by $\mu$ the law of $b(\omega)=e^{Q(\omega_{0})}$, we have
\[
\begin{split}
\int \log^{+} \log (1+|x|)\, d\mu(x)&=\frac{1}{\pi}\int_{\mathbb{R}} \frac{\log^{+} \log(1+e^{x})}{1+x^2}\, dx<\infty.
\end{split}
\]
Therefore, Theorem $\ref{Eltongener}$ and Corollary \ref{coro:letac1} apply. In contrast, note that 
\[
\int \log^+b(\omega) \, d\mathbb P_m(\omega) = \int Q^+(\omega_0) \, dm(\omega_0) = + \infty
\]
so the classical fixed-point displacement condition fails in the Euclidean metric.

This example illustrates Theorem~\ref{tb}. Since $\mathcal{T}\mu = \mbox{Law}(b + L(b'))$, where $b,b'$ are independent with law \(\mu\). Let $\Psi(x,y) = (x,L(y)+x)$ and consider $\gamma = \Psi_*(\mu \otimes \mu)$,  then $\gamma$ is a coupling, hence, for any $0 < q < 1$
\[
\begin{split}
W_q(\mathcal{T}\mu,\mu)
&\le \int |x+L(y)-x|^q \, d(\mu \otimes \mu)(x,y)
=\frac{1}{2^q\pi}\int_{\mathbb{R}} \frac{(\log(1+e^x))^q}{1+x^2}\,dx < \infty.
\end{split}
\]
However, the displacement condition does not apply: for every $x_0\in\mathbb{R}$ and every $q>0$:
$$
\int_{\mathbb{R}}|f_b(x_0)-x_0|^q\,d\mu(b)=\infty.
$$
\end{example}

We now present another case where Theorem~\ref{tb} remains applicable, whereas Kloeckner's displacement condition fails for $q>1$.

\begin{example}\label{Example:toy}
Let $(I,\mathcal{B},m)$ and $Q$ be as in Example \ref{ex:translationsandcontractions}. Let $X\colon I\to \mathbb{R}$ be a random variable satisfying 
$$
\int |X|^q\, dm<\infty
$$ 
for some $q\geq1$ and set $\eta=X_{*}m$.
Consider the i.i.d. RIM $(f,\mathbb{P}_{m})$ with $f\colon I^{\mathbb{Z}}\times \mathbb{R}\to \mathbb{R}$  given by
\begin{equation}\label{toy}
f_\omega(x)=\alpha x+P(\omega_{0})
\end{equation}
where $0<\alpha<1$ and $P\colon I\to \mathbb{R}$ is a random variable with law equal to $\mbox{Cauchy}(0,1) * \eta$  where $*$ denotes the convolution. We take as a reference measure $\mu=\mbox{Cauchy}(0,\frac{1}{1-\alpha})$. We claim that
\[
W_1(\mu,\mathcal{T}\mu)<\infty.
\]
To this end, let $S\colon I \to \mathbb{R}$ be a random variable with law $\mu$.  We have that the random variable 
$$
\omega\mapsto f_{\omega}(S(\omega_{-1}))
$$
has law equal to $\mathcal{T}\mu$. On the other hand,
$$
f_{\omega}(S(\omega_{-1}))=\alpha S(\omega_{-1})+P(\omega_{0}).
$$
Note that the random variables $\omega\mapsto P(\omega_{0})$ and $\omega\mapsto Q(\omega_{0})+X(\omega_{1})$ have the same law. Hence, 
$$
f_{\omega}(S(\omega_{-1}))\overset{d}{=} \alpha S(\omega_{-1})+Q(\omega_{0})+X(\omega_{1})
$$
where $\overset{d}{=}$ denotes equality in law. Using the stability of Cauchy distributions under independent linear combinations, we have that $\alpha S(\omega_{-1})+Q(\omega_{0})$ has law
\[
\operatorname{Cauchy}
\left(0,
\frac{\alpha}{1-\alpha}+1
\right)
=
\operatorname{Cauchy}
\left(0,
\frac{1}{1-\alpha}
\right).
\]
Therefore, 
$$
f_{\omega}(S(\omega_{-1}))\overset{d}{=} S(\omega_{-1})+X(\omega_{1})
$$

The law of $(S(\omega_{-1}),S(\omega_{-1})+X(\omega_{1}))$ is a coupling between $\mu$ and $\mathcal{T}\mu$, which implies that 
$$
W_{q}(\mu,\mathcal{T}\mu)^{q}\leq \int |X(\omega_{1})|^{q}\, \mathbb{P}_{m}(\omega)<\infty. 
$$
\end{example}

\noindent
{\bf Acknowledgements:} We thank Jeroen Lamb, Konstantinos Kourilous, and Matthias Wolfrum for enlightening discussions. This research was supported in part by FAPESP Grant No. 2023/13706-0. TP thanks CNPq 312287/2021-6 and the Serrapilheira Institute (Grant No. 270 Serra-1709-16124), and warmly thanks the Weierstrass Institute for their hospitality during multiple visits. TP also thanks the Humboldt Foundation for support through a Bessel Fellowship. BF thanks FAPESP for support through Grant No. 2025/09758-0.

\section{Proofs}\label{sec:proofs}
We begin with two uniqueness results showing that suitable synchronization
properties of random orbits imply uniqueness of invariant graphs and stationary
measures.

\begin{proposition}\label{prop:atmost1}
 Let $(f,\theta,\mathbb{P})$ be a stationary RIM. Assume that for $\mathbb{P}$-almost every $\omega$, we have 
$$
 \lim_{n\to \infty}d(f_{\omega}^{n}(x), f_{\omega}^{n}(y))=0
$$
for every $x,y\in M$. Then, there is at most one invariant graph. 
\end{proposition}

\begin{proof}
Let $\pi$ and $\alpha$ be invariant graphs. We will show that $\pi=\alpha$  $\mathbb P$-almost everywhere. Consider the bounded continuous function
\begin{equation}\label{Eq:hat_metric}
\hat{d}:M \times M \to \mathbb R, \quad \hat d(x,y)=\frac{d(x,y)}{1+d(x,y)}.
\end{equation}
Since $\alpha$ and $\pi$ are invariant maps, we have 
$$
\pi(\theta^{n}(\omega))=f_{\omega}^{n}(\pi(\omega)) \quad \quad \mbox{and} \quad \quad \alpha(\theta^{n}(\omega))=f_{\omega}^{n}(\alpha(\omega)) 
$$
for $\mathbb{P}$-almost every $\omega$. Thus, by the $\theta$-invariance of $\mathbb{P}$ we conclude that 
\[
\begin{split}
\int \hat 
d(\pi(\omega), \alpha(\omega) )\, d\mathbb{P}(\omega)
=
\int \hat d(
\pi(\theta^
{n}(\omega)),\alpha(\theta^{n}(\omega))
)
\, d\mathbb{P}(\omega)
=\int \hat d(f^{n}_{\omega}(\pi(\omega)),
f^{n}_{\omega}(\alpha(\omega)))\, d\mathbb{P}(\omega).
\end{split}
\]
By hypothesis, we have
$$
\lim_{n\to \infty}d(f_{\omega}^{n}(\pi(\omega)),f_{\omega}^{n}(\alpha(\omega)))=0
$$
for $\mathbb{P}$-almost every $\omega$, which implies that
$$
\lim_{n\to \infty}\hat d(f_{\omega}^{n}(\pi(\omega)),f_{\omega}^{n}(\alpha(\omega)))=0
$$
for $\mathbb{P}$-almost every $\omega$. Since $\hat d$ is bounded, applying the dominated convergence theorem we get 
$$
\int \hat d(\pi(\omega),\alpha(\omega))\, d\mathbb{P}(\omega)=0.
$$
Hence, $d(\pi(\omega),\alpha(\omega))=0$ for $\mathbb{P}$-almost every $\omega$.
\end{proof}

\begin{proposition}\label{prop:atmost2}
Let $(f,\mathbb{P}_{\nu})$ be an i.i.d. RIM, and assume that for every  $x,y\in M$, we have 
$$ 
\lim_{n\to \infty}d(f_{\omega}^{n}(x), f_{\omega}^{n}(y))=0
$$
for $\mathbb{P}$-almost every $\omega$. Then, there is at most one stationary measure. 
\end{proposition}

\begin{proof}
Assume that the set of stationary measures is non-empty. By the ergodic decomposition theorem, it is enough to prove that any two ergodic stationary measure coincide. Let $\mu_{1}$ and $\mu_{2}$ denote ergodic stationary measures. Let us show that $\mu_{1}=\mu_{2}$. 

Fix \(G:M \to \mathbb R\) a bounded Lipschitz map. By the Birkhoff ergodic theorem for Markov chains, for each $\ell=1,2$, for $\mathbb P$-almost every $\omega$ and $\mu_\ell$-almost every $x$ (typical points), we have 
\begin{equation}\label{eq:bir}
\lim_{n\to \infty} \frac{1}{n}\sum_{i=0}^{n-1} G(f_{\omega}^{i}(x))=\int G(x)\, d\mu_{\ell}(x).
\end{equation}

Let $x_{1}$ be a typical point for $\mu_{1}$ and let $x_{2}$ be a typical point for $\mu_{2}$. By hypothesis, for $\mathbb{P}$-almost every $\omega$, we have 
\begin{equation}\label{eq:sync}
\lim_{n\to \infty
}d(f^{n}_{\omega}(x_{1}),f_{\omega}^{n}(x_{2}))=0.
\end{equation}

Since on a probability space, countable intersection of measurable sets of full measure has full measure, we can take $\omega$ satisfying \eqref{eq:sync} and such that for every $\ell=1,2$,   \eqref{eq:bir} is satisfied for $x=x_{\ell}$. This implies that 
$$
\int G(x)\, d\mu_{1}(x)=\lim_{n\to \infty} \frac{1}{n}\sum_{i=0}^{n-1} G(f_{\omega}^{i}(x_{1}))=\lim_{n\to \infty} \frac{1}{n}\sum_{i=0}^{n-1} G(f_{\omega}^{i}(x_{2}))=\int G(x)\, d\mu_{2}(x).
$$
Since $G$ was an arbitrary bounded Lipschitz function, $\mu_{1}=\mu_{2}$.
\end{proof}

\begin{remark} \emph{
Under the assumptions of Proposition~\ref{prop:atmost2}, and without additional hypotheses, one cannot conclude that there is at most one invariant graph. See \cite{Kleptsyn2004} for an example of an i.i.d. random iteration of circle homeomorphisms with two invariant graphs.}   
\end{remark}

\begin{remark}[\emph{Synchronization on average}]
\emph{ Both propositions remain valid if the synchronization condition
\[
d(f_\omega^n(x),f_\omega^n(y))\to0
\]
is replaced by synchronization on average,
$$
\lim_{n\to \infty} \frac{1}{n}\sum_{i=0}^{n-1} d(f_{\omega}^{i}(x),f_{\omega}^{i}(y))=0,
$$
provided that the order of the quantifiers is preserved in each statement.}
\end{remark}

\subsection{Proof of Theorem \ref{Eltongener}}

Since the maximal Lyapunov exponent is negative, by Proposition \ref{prop:atmost1}, we need only to show that for every reference map $\varphi$, for $\mathbb P$-almost every $\omega$, the sequence
$$
\Gamma^{n}\varphi(\omega)=
f_{\theta^{-1}(\omega)} \circ \cdots \circ
f_{\theta^{-n}(\omega)}
\bigl(\varphi(\theta^{-n}(\omega))\bigr)
$$
converges and that its limit defines an invariant graph.

To prove convergence, it is enough to show that 
\begin{equation}\label{cauchy}
\sum_{n=0}^{\infty} d(\Gamma^{n}\varphi(\omega),\Gamma^{n+1}\varphi(\omega))<\infty 
\end{equation}
for $\mathbb{P}$-almost every $\omega$.
Indeed, \eqref{cauchy} implies that the sequence $(\Gamma^{n}\varphi(\omega))_{n\in \mathbb{N}}$ is a Cauchy sequence for $\mathbb{P}$-almost every $\omega$.
Set $D(\omega)=d(\varphi(\omega),\Gamma\varphi (\omega)). $ First, note that
\[
\begin{split}
d(\Gamma^{n}\varphi(\omega),\Gamma^{n+1}\varphi(\omega))
&\leq \|f_{\theta^{-1}(\omega)}\circ \dots\circ  f_{\theta^{-n}(\omega)}\| D(\theta^{-n}(\omega)).
\end{split}
\]
 
We estimate the two factors on the right-hand side separately. The first is controlled by the negative maximal Lyapunov exponent, while the second has subexponential growth by the generalized displacement condition.

The reverse-order iteration $f^{n}_{\theta^{-n}(\omega)}= f_{\theta^{-1}(\omega)} \circ \cdots\circ f_{\theta^{-n}(\omega)} $ also has a well-defined asymptotic exponential growth rate, and this rate coincides with the maximal Lyapunov exponent,
\begin{equation}\label{NExp}
\lambda(\omega)
=\lim_{n\to\infty}
\frac1n
\log\|f_{\omega}^n\|=
\lim_{n\to\infty}
\frac1n
\log\|f_{\theta^{-n}(\omega)}^n\|
\end{equation}
for $\mathbb P$-almost every $\omega$. See \cite{Elton1990} or  \cite[Lemma 5.4]{Stark1999}.
On the other hand, by hypothesis the random variable $\log ^{+}D$  is integrable, which implies that 
\begin{equation}\label{subexpodisplacement}
\lim_{n\to \infty}\frac{1}{n} \log^{+} D(\theta^{-n}(\omega))=0
\end{equation}
for $\mathbb{P}$-almost every $\omega$. This is a standard fact about general integrable functions and measure-preserving transformations, which follows from the Borel-Cantelli lemma.  
 
Fix $\omega$ satisfying \eqref{NExp} and \eqref{subexpodisplacement}. Since $\lambda(\omega)$ is negative, choose $\kappa(\omega)>0$ such that 
\[
\kappa(\omega)+\lambda(\omega)<0,
\] 
if \(-\infty<\lambda(\omega)\), otherwise choose any \(\kappa(\omega)>0\). Set
\[
\varepsilon(\omega) = \begin{cases} 
\kappa(\omega)
&\qquad\text{if} \quad \lambda(\omega)=-\infty\\
-(\kappa(\omega)+\lambda(\omega)) &\qquad\text{otherwise}
\end{cases}
\]

Then, there exists $n_{0} = n_{0}(\omega,\varepsilon)$ for which
\begin{equation}\label{negativeback}
\|f_{\theta^{-n}(\omega)}^{n}\|
\leq e^{-n\,\varepsilon(\omega)}
\end{equation}
holds for all $n \geq n_{0}$. Choose $\gamma>0$ so that $-\epsilon(\omega)+\gamma<0$. By \eqref{subexpodisplacement}, there exists $m_0$ such that 
$$
n\geq m_{0} \implies \frac{1}{n}\log^{+} D(\theta^{-n}(\omega))<\gamma.
$$
In particular, 
$$
n\geq m_{0} \implies  D(\theta^{-n}(\omega))<e^{n\gamma}.
$$
Therefore, for $n\geq \max\{n_{0},m_{0}\}$ we have  
$$
d(\Gamma^{n}\varphi(\omega),\Gamma^{n+1}\varphi(\omega))\leq e^{-n\epsilon(\omega)}e^{n\gamma}=e^{n(-\epsilon(\omega)+\gamma)}=e^{nu}
$$
where $u=-\epsilon(\omega)+\gamma$. Since $u<0$, it follows that
$$
\sum_{n=0}^{\infty}d(\Gamma^{n}\varphi(\omega),\Gamma^{n+1}\varphi(\omega))<\infty.
$$

Let  $\pi\colon \Omega\to M$ be a measurable map such that 
\begin{equation}\label{refgraph}
\pi(\omega)=\lim_{n\to \infty} \Gamma^{n}\varphi(\omega)
\end{equation}
for $\mathbb{P}$-almost every $\omega$. By the continuity of $f_\omega$, we have 
\[
\begin{split}
f_{\omega}(\pi(\omega))&=f_{\omega}(\lim_{n\to\infty}
f_{\theta^{-1}(\omega)} \circ \cdots \circ
f_{\theta^{-n}(\omega)}
(\varphi(\theta^{-n}(\omega))))\\
&=\lim_{n\to \infty}f_{\omega}\circ f_{\theta^{-1}(\omega)}\circ \dots \circ f_{\theta^{-n}(\omega)}(\varphi(\theta^{-n}(\omega)))
=\pi(\theta(\omega)).
\end{split}
\]
This proves that \(\pi\) is an invariant graph. Moreover, by Proposition~\ref{prop:atmost1}, the invariant graph is unique.
Consequently, the limit above is independent of the chosen reference map
$\varphi$.

It remains to prove convergence in law. Fix a reference map $\varphi$ and let
$G\colon M\to\mathbb R$ be bounded and continuous. Since $\mathbb{P}$ is $\theta$-invariant and $\varphi$ is a reference map, we have
 \[
 \begin{split}
\lim_{n\to \infty} 
\int G(f_{\omega}^{n}(\varphi(\omega)))
\, d\mathbb{P}(\omega)
=\lim_{n\to \infty}\int G(\Gamma^{n}\varphi(\omega))\, d\mathbb{P}(\omega)
=\int G(x)\, d\pi_{*}\mathbb{P}(x).
 \end{split}
\]

This completes the proof of the theorem.

\hfill $\square$

\subsection{Proof of Corollary~\ref{coro:letac1}}

We use the following proposition, which may be viewed as a generalization of the main result in \cite{Letac1986}.

\begin{proposition}[Generalized Letac principle]\label{prop:genletac}
Consider an i.i.d. RIM $(f,\mathbb{P}_{\nu})$ on a metric space $M$. Let $\psi\colon \Omega\to M$ be an $\mathcal{F}_{-\infty}$-measurable map and set $\mu=\psi_{*}\mathbb{P}_{\nu}$. Then
\begin{itemize}
\item[(i)] The law of the random element $\omega\mapsto f_{\omega}^{n}(\psi(\omega))$ is given by $\mathcal{T}^{n}\mu$.
\item[(ii)] Assuming that the limit 
$$
\rho(\omega)=\lim_{n\to \infty} \Gamma^{n} \psi(\omega) 
$$
exists for $\mathbb{P}_{\nu}$-almost every $\omega$, we have 
$$
\mathcal{T}^{n}(\psi_{*}\mathbb{P}_{\nu})\to \rho_{*}\mathbb{P}_{\nu}
$$
where the convergence is in the weak$^{*}$-topology. 
\end{itemize}

\end{proposition}
\begin{proof}

Given an integrable random variable $X\colon \Omega\to \mathbb{R}$ and a $\sigma$-algebra $\mathcal{G}$, we shall denote by $\mathbb{E}(X|\mathcal{G})$ the conditional expectation of $X$ with respect to $\mathcal{G}$. As usual $\mathbb{E}(X)$ denotes the expectation of $X$.  To prove item (i), we need the following lemma.

\begin{lemma}\label{lemmaletac}
Let $\psi\colon\Omega\to M$ be an $\mathcal F_{-\infty}$-measurable map, and
let
\[
g\colon\Omega\times M\to\mathbb R
\]
be a bounded $\mathcal F_{+\infty}\otimes\mathcal B(M)$-measurable function.
Then
\[
\mathbb E\bigl(g(\cdot,\psi(\cdot))\mid\mathcal F_{-\infty}\bigr)(\omega)
=
\mathbb E\bigl(g(\cdot,\psi(\omega))\bigr)
\]
for $\mathbb P$-almost every $\omega$.
\end{lemma}

\begin{proof}
Let $\mathcal H$ be the class of bounded
$\mathcal F_{+\infty}\otimes\mathcal B(M)$-measurable functions
$g\colon\Omega\times M\to\mathbb R$ such that
\[
\mathbb E\bigl(g(\cdot,\psi(\cdot))\mid\mathcal F_{-\infty}\bigr)(\omega)
=
\mathbb E\bigl(g(\cdot,\psi(\omega))\bigr)
\]
for $\mathbb P$-almost every $\omega$. We will show that $\mathcal H$ contains all bounded $\mathcal F_{+\infty}\otimes\mathcal B(M)$-measurable functions.

First, observe that $\mathcal H$ is a monotone vector class. Indeed, if $g,h\in\mathcal H$ and $c\in\mathbb R$, then $g+h\in\mathcal H$ and $cg\in\mathcal H$, by linearity of expectation and conditional expectation. Moreover, if $g_n\in\mathcal H$ are nonnegative and increase pointwise to a bounded function $g$, then $g\in\mathcal H$, by the monotone convergence theorem applied both to expectations and conditional expectations.

Let $\mathcal C$ be the class of rectangles
\[
A^+\times B,
\qquad
A^+\in\mathcal F_{+\infty},\quad B\in\mathcal B(M).
\]
This is a $\pi$-system generating $\mathcal F_{+\infty}\otimes\mathcal B(M)$. By the monotone class theorem, it is enough to prove that $\mathds 1_{A^+\times B}\in\mathcal H$ for every $A^+\in\mathcal F_{+\infty}$ and every $B\in\mathcal B(M)$. Fix such a rectangle and set
\[
g=\mathds 1_{A^+\times B}.
\]
Let $A^-\in\mathcal F_{-\infty}$. Since $\psi$ is $\mathcal F_{-\infty}$-measurable,  $\mathds 1_B(\psi(\omega))$ is $\mathcal F_{-\infty}$-measurable. Hence
\[
\mathds 1_{A^-}(\omega)\mathds 1_B(\psi(\omega))
\]
is $\mathcal F_{-\infty}$-measurable. On the other hand, $\mathds 1_{A^+}$ is $\mathcal F_{+\infty}$-measurable. Since $\mathcal F_{-\infty}$ and $\mathcal F_{+\infty}$ are independent, we obtain
\[
\begin{aligned}
\int_{A^-} g(\omega,\psi(\omega))\,d\mathbb P(\omega)
=
\int
\mathds 1_{A^-}(\omega)
\mathds 1_{A^+}(\omega)
\mathds 1_B(\psi(\omega))\,d\mathbb P(\omega)
=
\mathbb P(A^+)
\int_{A^-}
\mathds 1_B(\psi(\omega))\,d\mathbb P(\omega).
\end{aligned}
\]
Moreover, for every $\omega$,
\[
\mathbb E\bigl(g(\cdot,\psi(\omega))\bigr)
=
\int
\mathds 1_{A^+}(\vartheta)
\mathds 1_B(\psi(\omega))\,d\mathbb P(\vartheta)
=
\mathbb P(A^+)\mathds 1_B(\psi(\omega)).
\]
Therefore,
\[
\begin{aligned}
\int_{A^-}
\mathbb E\bigl(g(\cdot,\psi(\omega))\bigr)\,d\mathbb P(\omega)
=
\mathbb P(A^+)
\int_{A^-}
\mathds 1_B(\psi(\omega))\,d\mathbb P(\omega)
=
\int_{A^-} g(\omega,\psi(\omega))\,d\mathbb P(\omega).
\end{aligned}
\]
Since this holds for every $A^-\in\mathcal F_{-\infty}$, we conclude that
\[
\mathbb E\bigl(g(\cdot,\psi(\cdot))\mid\mathcal F_{-\infty}\bigr)(\omega)
=
\mathbb E\bigl(g(\cdot,\psi(\omega))\bigr)
\]
for $\mathbb P$-almost every $\omega$. Thus $\mathds 1_{A^+\times B}\in\mathcal H$, and the monotone class theorem completes the proof.
\end{proof}

We first prove item~(i). For every $n$, we define $g_{n}\colon \Omega\times M\to \mathbb{R}$ by 
$$
g_{n}(\omega,x)=G(f_{\omega}^{n}(x)).
$$
This function is a bounded $\mathcal{F}_{+\infty}\otimes \mathcal{B}(M)$- measurable  function. From the lemma we obtain
$$
\mathbb{E}(g_{n}(\cdot,\psi(\cdot))|\mathcal{F}_{-\infty})(\omega)=\mathbb{E}(g_{n}(\cdot, \psi(\omega))).
$$
In particular, since $\Omega\in \mathcal{F}_{-\infty}$, we have 
$$
\int g_{n}(\omega,\psi(\omega))\, d\mathbb{P}(\omega)=\int \mathbb{E}(g_{n}(\cdot, \psi(\omega)))\, d\mathbb{P}(\omega)
$$
or equivalently:
$$
\int G(f_{\omega}^{n}(\psi(\omega)))\, d\mathbb{P}(\omega)=\int \mathbb{E}(G(f^{n}_{(\cdot)}( \psi(\omega))))\, d\mathbb{P}(\omega).
$$
It follows from the Lemma \ref{lemmaletac} that
\[
\begin{split}
\int G(x)\, d\mathcal{T}^{n}(\psi_{*}\mathbb{P})(x)
=\int \mathbb{E}(G(f_{(\cdot)}^{n}(\psi(\omega))))\, d\mathbb{P}(\omega)
=\int G(f_{\omega}^{n}(\psi(\omega)))\, d\mathbb{P}(\omega).
\end{split}
\]
This completes the proof of item (i). 
We next prove item~(ii). Let $G\colon M\to \mathbb{R}$ be a continuous bounded function. Note that
$$
\int G(f_{\omega}^{n}(\psi(\omega)))\, d\mathbb{P}(\omega)=\int G(\Gamma^{n}\psi(\omega))\, d\mathbb{P}(\omega).
$$
It follows from item~(i)  that 
\[
\lim_{n\to\infty}
\int G(x)\,d\mathcal T^n(\psi_*\mathbb P)(x)
=
\int G(x)\,d\rho_*\mathbb P(x).
\]
\end{proof}

To conclude the proof of the corollary, it remains to establish the existence of a stationary measure. Once this is done, Corollary~\ref{coro:letac1} follows from Proposition~\ref{prop:genletac}, Theorem~\ref{Eltongener}, and Proposition~\ref{prop:atmost2}. Since $\mathcal T$ is continuous, one could try to apply item~(ii) of Proposition~\ref{prop:genletac} directly. However, our assumptions guarantee only the existence of a reference map, which need not be $\mathcal F_{-\infty}$-measurable. We therefore prove directly that the invariant map $\pi$ is $\mathcal F_{-\infty}$-measurable. Then, stationarity of $\pi_{*}\mathbb{P}$ will follow from the $\theta$-invariance of  $\mathbb{P}$, the invariance equation
 $$
f_{\omega}(\pi(\omega))=\pi(\sigma(\omega))\quad \quad \mbox{a.e.}
$$
and item $(i)$ of Proposition \ref{prop:genletac}.

Consider the function 
$$
\hat d\colon M\times M\to [0,1], \quad \hat d(x,y)=\frac{d(x,y)}{1+d(x,y)}.
$$
Fix $p\in M$ and note that the negative maximal Lyapunov exponent and the invariance of the map $\pi$ imply that 
$$
d(\pi(\sigma^{n}(\omega)),f_{\omega}^{n}(p))=d(f_{\omega}^{n}(\pi(\omega)),f_{\omega}^{n}(p))\to 0.
$$
In particular, 
$$
\hat d(\pi(\sigma^{n}(\omega)),f_{\omega}^{n}(p))\to 0
$$
Now, by the dominated convergence theorem and the $\theta$-invariance of $\mathbb{P}$ we have 
$$
\int \hat d(\pi(\sigma^{n}(\omega)),f_{\omega}^{n}(p))
\, d\mathbb P(\omega)
=\int \hat d(\pi(\omega),f_{\omega_{-1}}\circ \dots \circ f_{\omega_{-n}}(p))\, d\mathbb{P}(\omega)\to 0
$$
Then, we can extract a subsequence $(n_{k})$ such that 
$$
\hat d(\pi(\omega),f_{\omega_{-1}}\circ \dots \circ f_{\omega_{-n_{k}}}(p))\to 0
$$
for $\mathbb{P}$-almost every $\omega$,
which implies that 
$$
\pi(\omega)=\lim_{k\to \infty}  f_{\omega_{-1}}\circ \dots \circ f_{\omega_{-n_{k}}}(p).
$$
This proves that $\pi$ is $\mathcal{F}_{-\infty}$-measurable.

\subsection{Proof of Theorem \ref{tb}}
Let $\mathcal{T}$ be the Markov operator associated with an i.i.d. RIM $(f,\mathbb{P}_{\nu})$ that is $r$-contracting on $L^{q}$-average between points. For every $\eta\in\mathcal{P}_{1}(M)$, we denote by $\mathcal{P}^{q}_{\eta}$ the finite-distance component of $\eta$ with respect to the extended Wasserstein metric $W_q$. Namely,
\[
\mathcal{P}^{q}_{\eta}
=
\left\{
\mu\in\mathcal{P}(M)
\colon
W_{q}(\mu,\eta)<\infty
\right\}.
\]
For every $\eta$ satisfying $W_{q}(\eta,\mathcal{T}\eta)<\infty$ we will prove that $\mathcal{T}$ is a contraction on the space $\mathcal{P}^{q}_{\eta}$.

We start by proving that for every $\mu_{1},\mu_{2}\in \mathcal{P}_{1}(M)$, we have
\[
W_{q}(\mathcal{T}\mu_{1},\mathcal{T}\mu_{2})
\leq r^{\min\{1,1/q\}} W_{q}(\mu_{1},\mu_{2}). 
\]

Given probability measures $\mu_{1}$ and $\mu_{2}$, let $\Upsilon$ be a transport plan between $\mu_{1}$ and $\mu_{2}$.  Recall that $\Omega=E^{\mathbb{Z}}$. For every $\alpha\in E$,  let $\hat f_{\alpha}\colon M\times M\to M\times M$ be the measurable map defined by 
\[
\hat f_{\alpha}(x,y)=(f_{\alpha}(x),f_{\alpha}(y)).
\]
We claim that  the Borel probability measure on $M\times M$ given by
$$
\mathbf{T}\Upsilon=\int \hat f_{\alpha*}\Upsilon\, d\nu(\alpha)
$$
is a transport plan between $\mathcal{T}\mu_{1}$ and $\mathcal{T}\mu_{2}$. In fact, for every Borel measurable set $A$ we have
\[
\begin{split}
	\mathbf{T}\Upsilon( A\times M)
	&=\int \hat f_{\alpha*}\Upsilon(A\times M)d\nu(\alpha)
	=\int \Upsilon (f_{\alpha}^{-1}(A)\times M)d\nu(\alpha)\\
	&=\int \mu_{1}(f_{\alpha}^{-1}(A))\, d\nu(\alpha)
	=\mathcal{T}\mu_1(A).
\end{split}
\]

In the same way, we can prove that $\mathbf{T}{\Upsilon}(M\times B)=\mathcal{T}\mu_{2}(B)$ for every Borel measurable set $B$.

We need the following lemma.
\begin{lemma}
For every Borel measurable map $g\colon M\times M\to [0,\infty]$ we have 
$$
\int g(x,y)\, d\mathbf{T}\Upsilon(x,y)=\int \int g(\hat f_{\alpha}(x,y))\, d\nu(\alpha)\, d\Upsilon(x,y)
$$
\end{lemma}
\begin{proof}
Let $g=\mathds{1}_{A\times B}$ be the characteristic function of a measurable rectangle $A\times B \subset M\times M$. By Fubini's theorem, we have
\[ 
\begin{split}
\int\int \mathds{1}_{A\times B}(\hat{f}_{\alpha}(x,y))\, d\nu(\alpha)\, d\Upsilon(x,y)
&= 
\int\int \mathds{1}_{A\times B}(\hat{f}_{\alpha}(x,y))\,d\Upsilon(x,y) d\nu(\alpha)\\
&=
\int \Upsilon(\hat{f}_{\alpha}^{-1}(A\times B))\, d\nu(\alpha)=\int \hat{f}_{\alpha*}\Upsilon(A\times B)\, d\nu(\alpha)
\\
&=\mathbf{T}\Upsilon(A\times B).
\end{split}
\]

The result follows from the Monotone Class Theorem. 
\end{proof}

Now, using the RIM $(f,\sigma,\mathbb{P}_{\nu})$ is a $r$-contracting on $L^{q}$-average and by the previous lemma, we have
\[
\begin{split}
\inf_{\bar \Upsilon\in\Gamma(\mathcal{T}\mu_{1},\mathcal{T}\mu_{2})}
\int d(x,y)^{q}\,d\bar\Upsilon(x,y)&\leq \int d(x,y)^{q}d\mathbf{T}\Upsilon(x,y)\\
	&=\int\left(\int d(f_{\alpha}(x),f_{\alpha}(y))^q\, d\nu(\alpha)\right)\, d\Upsilon(x,y) \\
	&\leq r \int  d(x,y)^{q}\, d\Upsilon(x,y).
\end{split}
\]

Since $\Upsilon$ is an arbitrary transport plan between $\mu_{1}$ and $\mu_{2}$, we conclude that 
$$
\inf_{\bar \Upsilon\in\Gamma(\mathcal{T}\mu_{1},\mathcal{T}\mu_{2})}
\int d(x,y)^{q}\,d\bar\Upsilon(x,y)\leq r \inf_{\Upsilon\in \Gamma(\mu_{1},\mu_{2})} \int d(x,y)^{q}\,d\Upsilon(x,y)
$$
and hence 
$$
W_q(\mathcal{T}\mu, \mathcal{T}\nu) \leq r^{\min\{1,\frac{1}{q}\}} \, W_q(\mu, \nu).
$$

We claim that $\mathcal{P}^{q}_{\eta}$ is invariant under the Markov operator. Let \(\mu \in \mathcal{P}^q_\eta\). We have
\[
W_q(\mathcal{T}\mu, \eta) 
\leq W_q(\mathcal{T}\mu, \mathcal{T}\eta) + W_q(\mathcal{T}\eta, \eta) 
\leq r^{\min\{1,1/q\}} W_q(\mu, \eta) + W_q(\mathcal{T}\eta, \eta) < \infty.
\]

To establish the existence of the stationary measure in Theorem~\ref{tb}, we need the following proposition. Its proof follows the same argument as in the case where $\eta$ is a Dirac measure, as presented in~\cite{Villani2009}, and is therefore omitted here.

\begin{proposition}
Let \(M\) be a Polish space, then the metric space $(\mathcal{P}_{\eta}^{q},W_{q})$ is complete.
\end{proposition}

Now, the existence of stationary measure follows from the Banach fixed point theorem. Notice  that the Banach fixed point theorem  gives existence and uniqueness only within the $W_q$-finite component of $\eta$. A priori, this does not exclude the existence of stationary measures outside the set $\mathcal{P}_{\eta}^{q}$.  Let us prove the uniqueness. To this end, we will apply  Proposition \ref{prop:atmost2}.

To do that, we need to show that every i.i.d. RIM that is  $r$-contracting on $L^{q}$-average between points  synchronizes.  This is proved in the next proposition. 
 
\begin{proposition}
Let  $(f,\mathbb{P}_{\nu})$  be an i.i.d. RIM that is $r$-contracting on $L^{q}$-average between points. Then, there is  $\lambda\in (0,1)$ such that for every $x, y\in M$,  there is a measurable map $c_{x,y}\colon \Omega\to [0,\infty)$  satisfying 
$$
d(f_\omega^n(x),f_\omega^n(y) )\leq c_{x,y}(\omega)\lambda^{n}d(x,y), \quad \quad \quad n\geq 1
$$
for $\mathbb{P}$-almost every $\omega$.
\end{proposition}

\begin{proof}
Let  $(f,\mathbb{P}_{\nu})$  be an i.i.d. RIM
that is $r$-contracting on $L^{q}$-average between points, that is, 
\[
\int d(f_\alpha(x), f_\alpha(y))^q d\nu(\alpha) \leq r d(x,y)^q \quad \mbox{for all } x,y \in M.
\]

Fix $x,y\in M$ with $x\neq y$. Define \(D_n(\omega) = d(f_\omega^n(x),f_\omega^n(y) )^q\), we claim that 
\[
\mathbb E(D_{n}) \leq r^n \, d(x,y)^{q} \quad \mbox{for all }x,y \in M
\]
for every $n$. 

Indeed, assume this statement is true for some $n\geq 1$. We have 
\[
\begin{split}
\mathbb{E}(D_{n+1})&=\int d(f_{\omega_{n}}(f_{\omega}^{n}(x)),f_{\omega_{n}}(f_{\omega}^{n}(y)))^{q}\, d\mathbb{P}(\omega)
\\
&=\int \int d(f_{\omega_{n}}(f_{\omega}^{n}(x)),f_{\omega_{n}}(f_{\omega}^{n}(y)))^{q}\,d\nu(\omega_{n})\,d\nu^{n}(\omega_{0},\dots,\omega_{n-1})\\
&\le \int rd(f^{n}_{\omega}(x),f^{n}_{\omega}(y))^{q}
\, \nu^{n}(\omega_{0},\dots, \omega_{n-1})\\
&=r \int d(f^{n}_{\omega}(x),f^{n}_{\omega}(y))^{q}\, d\mathbb{P}(\omega)=r \mathbb{E}(D_{n})\leq r^{n+1}d(x,y)^{q}.
\end{split}
\]
This proves the claim. 

Now, take a real number $a$ with $r<a<1$. By Tonelli's Theorem
\[
\int \sum_{n =0}^\infty a^{-n}\, d(x,y)^{-q} \, D_n(\omega) \, d\mathbb P(\omega) 
= \sum_{n =0}^\infty a^{-n}\, d(x,y)^{-q} \, \mathbb E(D_n) \leq  \sum_{n=0}^{\infty}  a^{-n}r^n < \infty.
\]
This implies that 
\[
c_{x,y}(\omega):=\left(\sum_{n =0}^\infty a^{-n}\, d(x,y)^{-q} \, D_n(\omega)\right)^\frac{1}{q}
 < \infty \quad \mbox{a.e. }
\]
Consequently,  for \(\mathbb P-\)almost every \(\omega\), and all $n\geq 1$, we have: 
$
d(f_\omega^n(x),f_\omega^n(y) )\leq c_{x,y}(\omega)a^{\frac{n}{q}}d(x,y).
$

\end{proof}
\bibliographystyle{plain}
\bibliography{referencias}

\end{document}